\documentclass[reqno]{amsart}
\usepackage{amssymb, bm, amsmath, latexsym, mathabx, dsfont,tikz, float, mathrsfs}
\usepackage{makecell, multirow, longtable, kotex}
\usepackage{enumitem}
\usepackage[capitalise]{cleveref}
\usepackage{ulem}
\usepackage{tikz}
\usepackage{kotex}
\usepackage{booktabs}

\renewcommand{\baselinestretch}{\baselinestretch}
\renewcommand{\baselinestretch}{1.1}
\numberwithin{equation}{section}
\numberwithin{table}{section}

\newcommand{\GL}{\mathrm{GL}}
\newcommand{\SL}{\mathrm{SL}}

\newcommand{\Q}{\mathbb{Q}}

\newcommand{\Z}{\mathbb{Z}}
\newcommand{\N}{\mathbb{N}}

\newcommand{\lcm}{\mathrm{lcm}}

\newcommand{\ra}{{\, \rightarrow \,}}

\newcommand{\z}{{\mathbb Z}}

\newcommand{\q}{{\mathbb Q}}

\newcommand{\n}{{\mathbb N}}

\newcommand{\of}[1]{{\mathcal{O}_F^{#1}}}
\newcommand{\Mod}[1]{\ (\mathrm{mod}\ #1)}

\newcommand{\slzeta}[2]{\zeta_{#1}^{\SL}\left(#2\right)}
\newcommand{\glzeta}[2]{\zeta_{#1}^{\GL}\left(#2\right)}
\newcommand{\Cl}{\mathrm{Cl}}

\usepackage[colorinlistoftodos]{todonotes}

	\author{Daejun Kim}
	\address{Department of Mathematics Education, Korea University,
		Seoul 02841, Republic of Korea}
	\email{daejunkim@korea.ac.kr}	
	\thanks{The first author was supported by the National Research Foundation of Korea(NRF) grant funded by the Korea government(MSIT) (RS-2024-00455692).}
	
	\author{Seok Hyeong Lee}
	\address{Center for Quantum Structures in Modules and Spaces, Seoul National University, Seoul 08826, Republic of Korea}
	\thanks{The second author was supported by the NRF grants No. 2020R1A5A1016126 and No. RS-2024-00462910.}
	\email{lshyeong@snu.ac.kr}

 \author{Seungjai Lee}
	\address{Department of Mathematics, Incheon National University, Incheon 22012, Republic of Korea}
	\email{seungjai.lee@inu.ac.kr}
\subjclass[2020]{11E12, 11E16, 11M41}
\newtheorem{thm}{Theorem}[section]
\newtheorem{lem}[thm]{Lemma}
\newtheorem{cor}[thm]{Corollary}
\newtheorem{prop}[thm]{Proposition}

\theoremstyle{definition}

\theoremstyle{remark}
\newtheorem{rmk}[thm]{Remark}

\numberwithin{equation}{section}

\title{Zeta functions of quadratic lattices of a hyperbolic plane}

\begin{document}

\begin{abstract}
    In this paper, we study the Dirichlet series that enumerates proper equivalence classes of full-rank sublattices of a given quadratic lattice in a hyperbolic plane---that is, a nondegenerate isotropic quadratic space of dimension $2$. We derive explicit formulas for the associated zeta functions and obtain a combinatorial way to compute them. Their analytic properties lead to the intriguing consequence that a large proportion of proper classes are one-lattice classes.
\end{abstract}
\maketitle
\allowdisplaybreaks
\section{Introduction}

In \cite{KLL}, the authors initiated the study of zeta functions associated with a quadratic lattice $L$, the Dirichlet series enumerating the proper equivalence classes of full-rank sublattices of $L$. To be more precise, let $L$ be a $\z$-lattice; that is, a free $\z$-module $L=\z e_1 + \cdots +\z e_n$ equipped with a quadratic form $Q:L\ra\z$ and the associated nondegenerate symmetric bilinear form $B:L\times L \ra \frac{1}{2}\z$, defined by $B(x,y)=\frac{1}{2}(Q(x+y)-Q(x)-Q(y))$ for any $x,y\in L$. The {\it discriminant} $d_L$ of $L$ is defined as the determinant of the {\it Gram matrix} $M_L:=(B(e_i,e_j))$ associated with $L$ and a basis $\{e_i\}$. We write $L\cong M_L$ to indicate that $L$ has a basis with Gram matrix $M_L$. We may naturally extend the quadratic form and the bilinear form to consider the quadratic space $V=\q L$ on which $L$ lies.

Two $\z$-lattices $L_1$ and $L_2$ on the same quadratic space $V$ share the same geometric properties if they are {\it isometric}; that is, there exists an {\it isometry} $\sigma\in O(V)$ such that $\sigma(L_1)=L_2$, where
\[
O(V)=\{\sigma\in \GL(V) : B(\sigma x,\sigma y)=B(x,y) \text{ for all }x,y\in V\}.
\]
is the {\it orthogonal group} of $V$, consisting of all isometries of $V$. We further say that $L_1$ and $L_2$ are {\it properly isometric} if $\sigma(L_1)=L_2$ for some {\it proper isometry} $\sigma\in O^+(V):=O(V)\cap \mathrm{SL}(V)$.

We are interested in the number of non-isometric full-rank sublattices of a given $\z$-lattice $L$ and how they distinguish $\z$-lattices. For this, we consider for $m\in\n$, the set $\mathcal{L}_m:=\{ K \subseteq L \mid [L:K]=m\}$ of all $\z$-sublattices of $L$ of index $m$ and define the {\it zeta function of proper isometry classes of $L$}, denoted by $\slzeta{L}{s}$, and the {\it zeta function of isometry classes of $L$}, denoted by $ \glzeta{L}{s}$, to be the Dirichlet series 
\[
    \slzeta{L}{s}:=\sum_{m=1}^\infty a_m^+(L) m^{-s} \qquad \text{and} \qquad   \glzeta{L}{s}:=\sum_{m=1}^\infty a_m(L) m^{-s},
\]
where $a_m^+(L)$ is the number of properly non-isometric lattices in $\mathcal{L}_m$, and $a_m(L)$ is the number of non-isometric lattices in $\mathcal{L}_m$. In \cite{KLL}, the authors investigated $\slzeta{L}{s}$ and $\glzeta{L}{s}$ for positive definite binary lattices $L$ such that $-4d_L$ being the discriminant of $F=\q(\sqrt{-4d_L})$. It turns out that the formulas for $\glzeta{L}{s}$ are complicated due to the presence of infinitely many improper isometries of $V$. 

On the other hand, the formulas for $\slzeta{L}{s}$ are rather simple and have nicer properties. Let $\zeta_{\z^2}(s)=\sum_{K\subseteq \z^2} \lvert \z^2 :K\rvert^{-s}$ denote the Dirichlet series enumerating all full-rank $\z$-sublattices $K\subseteq\z^2$. It is known that $\zeta_{\z^2}(s)=\zeta(s)\zeta(s-1)$, where $\zeta(s)$ is the Riemann zeta function (the book by Lubotzky and Segal \cite{LS} contains five different proofs for this), and \cite[Theorem 1.1]{KLL} states that for such $L$,
\[
\slzeta{L}{s} = \frac{2}{|\of{\times}|} \frac{\zeta_{\Z^2}(s)}{\zeta_{F}(s)} \left( \sum_{n=1}^{\infty} \frac{ \# \{ [J]\in \Cl_F: J \le \of{}, \, N(J)=n\} }{n^{s}} \right) +  \frac{|\of{\times}| - 2}{|\of{\times}|}  \left(\prod_{p\,\text{split}} (1-p^{-s}) \right) \zeta_{F}(s).
\]
Here $\of{}$, $\of{\times}$, $\Cl_F$ are the ring of integers, the group of units, the ideal class group of $F$, respectively; and $\zeta_F(s)$ is the Dedekind zeta function of $F$. As a result, such $L$ having the same discriminant share the same $\slzeta{L}{s}$ (see \cite[Corollary 1.2]{KLL}). It is natural to ask whether similar results can be obtained for indefinite binary lattices. 

\subsection{Main results and organizations}
In this paper, we investigate $\slzeta{L}{s}$ for indefinite binary $\z$-lattices $L$ on a {\it hyperbolic plane}---a $2$-dimensional quadratic space $V$ such that $Q(v)=0$ for some $v\in V$. Note that every hyperbolic plane contains a $\z$-lattice
\[
\mathbb{H}\cong\begin{pmatrix} 0&\frac{1}{2}\\ \frac{1}{2}&0 \end{pmatrix}.
\]
Moreover, since $\slzeta{L}{s}$ is invariant under scaling (see \cite[Remark 2.1]{KLL}) and isometric lattices share the same zeta function, it suffices to consider the sublattices $L\subseteq\mathbb{H}$ up to isometry. Lemma \ref{lem:Gram-Matrix-L} and Proposition \ref{prop:cond-proper-iso} demonstrate this reduction. 
As a result, we may restrict our attention to those $L$ with $\mathfrak{n}L=\z$ in proving the results, where the {\it norm ideal} $\mathfrak{n}L$ of $L$ is the $\z$-ideal generated by $Q(x)$ for all $x\in L$. It turns out that $\slzeta{L}{s}$ is invariant of $[\mathbb{H}:L](\mathfrak{n}L)^{-1}$ as follows.

\begin{thm} \label{thm:MainAnswer}
Let $L$ be a sublattice of $\mathbb{H}$ and let $B\in\n$ be such that $B\z=[\mathbb{H}:L](\mathfrak{n}L)^{-1}$. Then we have
\[
\slzeta{L}{s} =\frac{\zeta(s-1)}{\zeta(s)} \sum_{b \mid B} (B/b)^{-s}\sum_{m=1}^{\infty}  \frac{ \# \{d^2 \Mod {b} : d \mid m,\, \gcd(d,b)=1 \} }{m^s}.
\]
\end{thm}
The proof of \cref{thm:MainAnswer} is given in \cref{sec:pf-of-MainAnswer} assuming several technical lemmas, and their proofs are provided in \cref{sec:lemmas}. From \cref{thm:MainAnswer}, one can observe the following:
\begin{cor}\label{cor:asymptotics}
    The series $\slzeta{L}{s}$ has abscissa of convergence of $\Re(s)=2$. It can be meromorphically continued to $\Re(s)>1$ and has a simple pole at $s=2$.
\end{cor}
\begin{proof}
The Dirichlet series
\[
\sum_{m=1}^{\infty} \frac{ \#\{d^2 \Mod{b} : d \mid m, \gcd(d,b)=1 \}}{m^s}
\]
is holomorphic in $\Re(s) >1$, as its coefficients are bounded by $\varphi(b)$. So everything on the right side of the equation of Theorem \ref{thm:MainAnswer} except $\zeta(s-1)$ is holomorphic in $\Re(s)>1$. This shows that $\slzeta{L}{s}$ and $\zeta(s-1)$ shares same convergence properties on the region $\Re(s)>1$.
\end{proof}
For $b\in\N$, let us write $U_b=(\Z/b\Z)^{\times}$ for the unit group of $\Z/b\Z$ and $U_b^2=\{a^2:a\in U_b\}$ for the set of squares of the units in $\Z/b\Z$. In Section \ref{sec:comb}, we will show that $\slzeta{L}{s}$ can also be described as:
\begin{thm}\label{thm:MainAnswer2}
Let $L$ be a sublattice of $\mathbb{H}$ and let $B\in\n$ be such that $B\z=[\mathbb{H}:L](\mathfrak{n}L)^{-1}$. We have
    \begin{align*}
    \slzeta{L}{s}     
   &=\zeta(s-1)\sum_{b \mid B} (B/b)^{-s}\left(  \lvert U_b^2\rvert  - \sum_{{\bm w}\in W} c_b({\bm w}) \prod_{u \in  U_b^2  } \frac{\mathrm{Sym}_{w_u}(\{p^{-s} : p \in \mathcal{P}_u^{(b)} \} )}{\prod_{p \in \mathcal{P}_u^{(b)}}(1-p^{-s})^{-1}} \right),
\end{align*}
where $W$, $c_b(\bm{w})$, $\mathrm{Sym}_k(X)$, $\mathcal{P}_u^{(b)}$ are defined in Section \ref{sec:comb}.
\end{thm}
This allows us to explicitly compute $\slzeta{L}{s}$  in Section \ref{sec:explicit.exm}. Utilizing these explicit formulas, we are able to explore the number of {\it one-lattice proper classes}, the proper classes that contain only one lattice. Let us consider the ratio
\[
    P_{L}^{+}(X):=\frac{s_{X}^{+}(L)}{s_{X}(\Z^2)},
\]
where
\[
    s_{X}(\Z^2):=\#\{\textrm{sublattices of $\Z^2$ of index $<X$}\} \quad \text{and} \quad s_{X}^{+}(L):=\sum_{m<X}a_{m}^{+}(L).
\]
One may observe that among the $s_X^+(L)$ proper classes, at least $2s_X^+(L)-s_X(\z^2)$ are one-lattice proper classes. Therefore the proportion is at least $2P_L^+(X)-1$. The data in Table \ref{table1} indicate that a large portion of the proper classes are one-lattice proper classes---equivalently, a large portion of the sublattices of $L$ are not properly isometric to any other sublattices.

Asymptotics of $s_X(\z^2)$ and $s_X^+(L)$ as $X\ra \infty$ can be obtained from the analytic properties of corresponding Dirichlet series $\zeta_{\Z^2}(s)=\zeta(s)\zeta(s-1)$  and $\slzeta{L}{s}$. Note that both of them have abscissa of convergence of $\Re(s)=2$ and simple poles at $s=2$. Hence standard Tauberian techniques (eg. see \cite[p.121]{Nark}) immediately give asymptotics of
\begin{align*}
    s_{X}(\Z^2)\sim \frac{\zeta(2)}{2}X^{2} \quad \text{and} \quad s_{X}^{+}(L)\sim\frac{1}{2} \left(\mathrm{Res}_{s=2}\slzeta{L}{s}\right) X^2
\end{align*}
as $X\rightarrow \infty$. Since $\lim_{s\ra2}(s-2)\zeta(s-1)=1$, we have
\[
\mathrm{Res}_{s=2}\slzeta{L}{s} = \lim_{s \rightarrow 2} (s-2) \slzeta{L}{s} = \lim_{s \rightarrow 2} \frac{\slzeta{L}{s}}{\zeta(s-1)}.
\]
Thus the value can be obtained by substituting $s=2$ in the explicit formulas for $\frac{\slzeta{L}{s}}{\zeta(s-1)}$ in Section \ref{sec:explicit.exm}. Computing the approximate value for the residue, we obtain the approximate values of
\[
r := \lim_{X \rightarrow \infty} P_L^{+}(X) = \frac{1}{\zeta(2)} \left(\mathrm{Res}_{s=2}\slzeta{L}{s}\right)
\]
as listed in Table \ref{table1} for those lattices $L\subseteq\mathbb{H}$ with $B\z=[\mathbb{H}:L] (\mathfrak{n}L)^{-1}$ for various $B\in\N$.

\begin{table}[b]
    \centering
        \label{table1}
    \begin{tabular}{c|c|c @{\hskip 2cm} c|c|c}
        \toprule
        $B$ & $r$ & $2r-1$ & $B$ & $r$ & $2r-1$ \\\hline
        1  & 0.6079 & 0.2159 & 9  & 0.9207 & 0.8415 \\
        2  & 0.7599 & 0.5198 & 10 & 0.9310 & 0.8619 \\
        3  & 0.6755 & 0.3510 & 12 & 0.8866 & 0.7731 \\
        4  & 0.7979 & 0.5959 & 14 & 0.9616 & 0.9231 \\
        5  & 0.8491 & 0.6981 & 15 & 0.8945 & 0.7890 \\
        6  & 0.8443 & 0.6887 & 18 & 0.9647 & 0.9293 \\
        7  & 0.9137 & 0.8274 & 20 & 0.9514 & 0.9029 \\
        8  & 0.8074 & 0.6148 & 24 & 0.8971 & 0.7942 \\
        \bottomrule
    \end{tabular}
    \medskip
    \caption{Proportions of proper isometry classes in $L$.}
    \end{table}
Other than \cite{KLL}, to the authors knowledge there is no known work in the literature that systematically
investigated and computed the zeta functions enumerating proper class of quadratic lattices and their various asymptotics. Together with \cite{KLL}, we hope this article provides a reasonable starting point for a
more general theory of zeta functions of quadratic lattices.
\subsection{Notations and conventions}
We introduce several notations and conventions. Let $\zeta(s)$ be the Riemann zeta function, and for a prime $p$ let \[\zeta_{p}(s):=\frac{1}{1-p^{-s}}\] be the local Riemann zeta function such that $\zeta(s)=\prod_{p,\,prime}\zeta_{p}(s)$ is the Euler product. For an integer $n$ and a prime $p$, let
\[
\nu_{p}(n)=\begin{cases}
\max\{k\in\N_{0} : p^{k}|n\}&\textrm{if }n\neq0,\\
\infty&\textrm{if }n=0,
\end{cases}
\]
be the $p$-adic valuation of an integer $n$. We also define the $p$-adic valuation of rational numbers to be 
\[\nu_{p}\left(\frac{r}{s}\right)=\nu_{p}(r)-\nu_{p}(s).\]
For $x,y$  positive integers, we denote by $x \mid y^{\infty}$ if $x \mid y^k$ for some $k\in\z$. Also define 
\[
\gcd(x, y^{\infty} ) = \prod_{p : \nu_p(y)>0} p^{\nu_p(x)}.
\]
For $\bm{d}=(d_1,\ldots,d_k)\in\n^k$, we write $\gcd(\bm{d})=\gcd(d_1,\ldots,d_k)$, $\lcm(\bm{d})=\lcm(d_1,\ldots,d_k)$, etc. for simplicity. We write $\sigma_{1}(n)=\sum_{d|n}d$ to denote the sum of divisors function. For $b\in\N$, we write $U_b=(\Z/b\Z)^{\times}$ for the unit group of $\Z/b\Z$ and $U_b^2=\{a^2:a\in U_b\}$ for the set of squares of the units in $\Z/b\Z$. 
Any unexplained notation and terminology on $\z$-lattices can be found in \cite{OM}. 
\section{Preliminaries}
\subsection{Preliminaries: discrete subgroups of $\Q/\Z$ and their cosets}
For $x\in\Q$ and $d\in\Z_{>0}$, let 
\[\left( x+\frac{1}{d}\Z\right)/\Z:=\left\{x+\Z, \left(x+\frac{1}{d}\right)+\Z,\ldots,\left(x+\frac{d-1}{d}\right)+\Z\right\}\subseteq\Q/\Z.\]
Finite subgroups of $\Q/\Z$ are all of the form $d^{-1}\Z / \Z$ for integers $d>0$, and their cosets are of form $\left( x + d^{-1} \Z \right) /\Z$ for $x \in \Q$.  It is obvious that $\# \left( (x+d^{-1} \Z)/ \Z \right) = d$.

\begin{lem}\label{lem:coset.intersection}
Let $d_1, \ldots, d_k$ be positive integers and let $x_1,\ldots,x_k\in\q$. Then we have
\[
\bigcap_{i=1}^{k} \frac{1}{d_i}\Z / \Z = \frac{1}{\gcd(\bm{d})}\Z / \Z.
\]
Moreover, if $\bigcap_{i=1}^{k} (x_i + d_i^{-1} \Z )/ \Z \neq \emptyset$, then there exists $x\in\q$ such that
\[
    \bigcap_{i=1}^{k} \left(x_i + \frac{1}{d_i} \Z \right)/ \Z = \left(x+  \frac{1}{\gcd(\bm{d})} \Z \right) /\Z.
\]
\end{lem}

\begin{proof}
    The first part follows by an induction argument on $k$. To prove the second part, note that if $x+\z \in \bigcap_{i=1}^k(x_i + d_i^{-1} \Z)/\Z$, then $(x_i + d_i^{-1} \Z)/\Z = (x + d_i^{-1} \Z)/\Z$ for all $i$. Thus we have
\begin{equation*}
\bigcap_{i=1}^{k}\left(x_i + \frac{1}{d_i} \Z\right)/\Z = \bigcap_{i=1}^{k}\left(x + \frac{1}{d_i}\Z\right)/\Z = x  + \bigcap_{i=1}^{k} \frac{1}{d_i} \Z/\Z = \left(x + \frac{1}{\gcd(\bm{d})} \Z\right) / \Z. \qedhere
\end{equation*}
\end{proof}

\subsection{Sublattices of $\mathbb{H}$}
Let us discuss the sublattices of $\mathbb{H}$.
\begin{lem}\label{lem:Gram-Matrix-L} 
    Any sublattice $L$ of $\mathbb{H}$ has a basis with the Gram matrix of the form
    \[
    \begin{pmatrix}
        nA & \frac{nB}{2}\\
        \frac{nB}{2} & 0
    \end{pmatrix} \quad \text{where } n>0, \ 1\le A \le B \text{ and } \gcd(A,B)=1.
    \]
\end{lem}
\begin{proof} 
The proof follows merely from the row reduction of integral matrices. Let $\mathbb{H}=\z e_1 + \z e_2$ with $Q(e_1)=Q(e_2)=0$ and $B(e_1,e_2)=\frac{1}{2}$.
    Note that one may take a basis of $L$ as $L=\z(\alpha e_1+\beta e_2)+\z(\gamma e_2)$
    for some $\alpha,\beta,\gamma\in\z$. Let $A'=\frac{\beta}{\gcd(\beta,\gamma)}$, $B=\frac{\gamma}{\gcd(\beta,\gamma)}$ and $n=\alpha\cdot\gcd(\beta,\gamma)$. Since $\gcd(A',B)=1$, there are $k\in\z$ and $1\le A \le B$ with $\gcd(A,B)=1$ such that $A'=Bk+A$. Then we have
    \[
        L=\z(\alpha e_1+(\beta-\gamma k)e_2)+\z(\gamma e_2)\cong \begin{pmatrix} nA & \frac{nB}{2}\\ \frac{nB}{2} & 0 \end{pmatrix}. \qedhere
    \]
\end{proof}
Next we describe an equivalent condition for two sublattices of $L\subseteq\mathbb{H}$ to be properly isometric.
\begin{prop}\label{prop:cond-proper-iso}
Let $L$ be a sublattice of $\mathbb{H}$ such that 
\[
L=\z e_1+\z e_2=\begin{pmatrix} A & B/2 \\ B/2 &0 \end{pmatrix}.
\]
Then two sublattices $K_i=\z(a_ie_1+b_ie_2) + \z(d_ie_2)$ of $L$ with $a_i,d_i\in\n$, $b_i\in\z$ and $i=1,2$ are properly isometric if and only if 
\begin{equation}\label{eqn:properly-equiv-cond}
a_1 d_1 = a_2 d_2 \quad \text{and} \quad \left(\frac{A}{B} \frac{a_1}{d_1} + \frac{b_1}{d_1}\right) - \left(\frac{A}{B} \frac{a_2}{d_2} + \frac{b_2}{d_2}\right) \in\z.
\end{equation}
\end{prop}

\begin{proof} Note that $K_i$ is a sublattice of $L$ of index $a_id_i$, hence of discriminant $a_i^2d_i^2 dL$ and that two isometric lattices have the same discriminant. Thus we may always assume that $a_1d_1=a_2d_2$. 

Let $V=\q L$. We first describe $O^+(V)$. Let $\sigma\in O^+(V)$ and let us correspond $\sigma$ with the $2\times 2$ matrix $T_\sigma=(t_{ij})$, where $t_{ij}\in\z$ satisfy $\sigma(e_i)=t_{1i} e_1+ t_{2i} e_2$ for $i=1,2$. Then we have
\[
\left\{ T_\sigma : \sigma \in O^+(V) \right\} = \left\{ T_\sigma\in\SL_2(\q) : T_\sigma \begin{pmatrix} A & \frac{B}{2} \\ \frac{B}{2} & 0\end{pmatrix} T_\sigma^t = \begin{pmatrix} A & \frac{B}{2} \\ \frac{B}{2} & 0\end{pmatrix} \right\} = \left\{ \begin{pmatrix} z &  \frac{A(1-z^2)}{Bz} \\0 & z^{-1} \end{pmatrix}  : z \in \Q^{\times}\right\}.
\]

Note that $K_1$ and $K_2$ are isometric if and only if there exists $\sigma\in O^+(V)$ such that $\sigma(K_1)=K_2$. In terms of matrix presentation, this is equivalent to the existence of  $z \in \Q^{\times}$ such that
\begin{equation}\label{eqn:isom-matrix}
\begin{pmatrix}
\frac{a_1}{a_2}z& -\frac{a_1b_2}{a_2d_2}z+\frac{a_1A(1-z^2)}{d_2Bz}+\frac{b_1}{d_2z}\\
0& \frac{d_1}{d_2}\frac{1}{z}
\end{pmatrix}
=\begin{pmatrix} a_1 & b_1 \\ 0 & d_1 \end{pmatrix} \begin{pmatrix} z & \frac{A(1-z^2)}{Bz} \\  0 & z^{-1} \end{pmatrix}
\begin{pmatrix} a_2 & b_2 \\ 0 & d_2 \end{pmatrix}^{-1} \in \GL_2(\Z).
\end{equation}
Using $a_1d_1=a_2d_2$, we have $\frac{a_1}{a_2} \in \Z$ and $\frac{d_1}{d_2} \frac{1}{z} = \left( \frac{a_1}{a_2} z \right)^{-1} \in \Z$, and hence $z=\pm \frac{a_2}{a_1}$. Plugging this into \eqref{eqn:isom-matrix}, the $(1,2)$-entry of the first matrix of \eqref{eqn:isom-matrix} is an integer if and only if the latter condition in \eqref{eqn:properly-equiv-cond} holds.
\end{proof}

\section{Proof of \cref{thm:MainAnswer}} \label{sec:pf-of-MainAnswer}

In this section, we provide the proof of \cref{thm:MainAnswer} assuming that several technical lemmas are true whose proofs will be given in the following section. 
\subsection{Evaluation of $a_m^+(L)$} Let $L$ be a sublattice of $\mathbb{H}$. By \cref{lem:Gram-Matrix-L} and since $\slzeta{L}{s}$ is invariant under the scaling of $L$, we may assume that 
\[
    L=\begin{pmatrix} A & B/2 \\ B/2 &0 \end{pmatrix}
\]
for some $A,B\in\z$ with $1\le A\le B$ and $\gcd(A,B)=1$. By \cref{prop:cond-proper-iso}, the proper classes of sublattices of $L$ of index $m$ are in one-to-one correspondence with elements of $\bigcup_{d \mid m} S_{m,d, A/B}$, where
\[
S_{m,d,A/B} := \left(\frac{A}{B} \frac{m/d}{d} + \frac{1}{d} \Z \right) / \Z.
\]
By inclusion-exclusion principle, we obtain
\begin{equation}\label{eqn:a(m)-inclusion-exclusion}
a_{L}^{+}(m) = \# \left( \bigcup_{d \mid m} S_{m,d,A/B} \right) = \sum_{k=1}^{\infty} (-1)^{k-1} \sum_{ \substack{d_1< d_2< \cdots < d_k\\ d_i | m}} \# \left( \bigcap_{i=1}^{k} S_{m,d_i, A/B} \right).
\end{equation} 
Note that \cref{lem:coset.intersection} shows that for $d_i\mid m$,
\begin{equation}\label{eqn:size-of-intersection}
  \bigcap_{i=1}^{k} S_{m,d_i, A/B}  \neq \emptyset \quad \Rightarrow \quad  \# \left( \bigcap_{i=1}^{k} S_{m,d_i, A/B} \right) = \gcd(\bm{d}).
\end{equation}
Later, it will be shown in \cref{prop:intersection.cond1} that for $d_i\mid m$, pairwise distinct, we have
\begin{equation}\label{eqn:intersection-nonempty-cond}
\bigcap_{i=1}^{k} S_{m,d_i, A/B} \neq \emptyset \quad \text{ if and only if } \quad  M(B;\bm{d}) \mid m,
\end{equation}
where $M(B;\bm{d})$ is defined to be the unique positive integer satisfying
\begin{equation}\label{eqn:defn-M}
M(B;\bm{d}) \Z = \left( \bigcap_{1 \le i \le k} d_i \Z \right) \cap \left( \bigcap_{1\le i<j \le k} \left( B \gcd(\bm{d}) \left(\frac{d_i}{d_j} - \frac{d_j}{d_i} \right)^{-1} \right) \Z \right).
\end{equation}
Plugging \eqref{eqn:size-of-intersection} and \eqref{eqn:intersection-nonempty-cond} into the last inner sum of \eqref{eqn:a(m)-inclusion-exclusion} and noting that $d_i\mid M(B;\bm{d})$, we have
\begin{equation}\label{eqn:a(m)-final}
a_m^+(L)= \sum_{k=1}^{\infty} (-1)^{k-1} \sum_{\substack{1\le d_1 < \cdots <d_k \\ M(B;\bm{d}) \mid m}} \gcd(\bm{d}).
\end{equation}

\subsection{Evaluation of $\slzeta{L}{s}$--Beginning} Now we plug \eqref{eqn:a(m)-final} into $\slzeta{L}{s}$ and obtain
\begin{align}
\zeta_L^{\SL}(s) &= \sum_{m=1}^{\infty} m^{-s} a_m^+(L) = \sum_{m=1}^{\infty}m^{-s} \sum_{k=1}^{\infty} (-1)^{k-1} \sum_{\substack{1 \le d_1 < \cdots < d_k \\ M(B;\bm{d}) \mid m} } \gcd(\bm{d}) \nonumber \\
&= \sum_{k=1}^{\infty} (-1)^{k-1} \sum_{1 \le d_1 < \cdots< d_k} \gcd(\bm{d}) \sum_{M(B;\bm{d}) \mid m} m^{-s} \nonumber \\ 
&= \sum_{k=1}^{\infty} (-1)^{k-1} \sum_{1 \le d_1 < \cdots< d_k} \gcd(\bm{d})  M(B;\bm{d})^{-s} \zeta(s). \label{eqn:slzeta-M-final}
\end{align}
In order to deduce the formula further, let us define 
\begin{equation}\label{eqn:defn-C}
C(B;\bm{d}) = \dfrac{B \cdot \mathrm{lcm}(\bm{d})}{M(B;\bm{d})}
\end{equation}
and consider the following equality that holds for any $N\in\n$ obtained by the M\"obius Inversion Formula
\begin{equation}\label{eqn:Ns-Mobius}
    N^s = \sum_{b \mid N}  F(b,s), \text{ where } F(b,s) :=  \sum_{d \mid b} d^s \mu\left(\frac{b}{d}\right).
\end{equation}
Since $\mathrm{lcm}(\bm{d}) \mid M(B;\bm{d}) \mid B\cdot\mathrm{lcm}(\bm{d})$ by \eqref{eqn:defn-M}, we have $C(B;\bm{d})\in\n$ and $C(B;\bm{d})\mid B$. Hence substituting  $M(B;\bm{d})$ with $\frac{B\cdot\mathrm{lcm}(\bm{d})}{C(B;\bm{d})}$ in \eqref{eqn:slzeta-M-final}, and then plugging \eqref{eqn:Ns-Mobius} with $N=C(B;\bm{d})$ into the resulting equation, we have
\begin{multline}\label{eqn:slzeta=SumofPhi}
\slzeta{L}{s} = \zeta(s) \left( \sum_{k=1}^{\infty} (-1)^{k-1} \sum_{1 \le d_1<\cdots<d_k} \frac{\gcd(\bm{d})}{B^s \mathrm{lcm}(\bm{d})^s} \sum_{b \mid C(B;\bm{d})} F(b,s) \right)\\
= \sum_{b \mid B} \frac{F(b,s)}{B^s}  \zeta(s) \left( \sum_{k=1}^{\infty} (-1)^{k-1}  \sum_{\substack{1 \le d_1 < \cdots < d_k \\ b \mid C(B;\bm{d})}} \frac{\gcd(\bm{d})}{\mathrm{lcm}(\bm{d})^s} \right) = \sum_{b \mid B} \frac{F(b,s)}{B^s} \Phi(B,b,s),
\end{multline}
where we put
\[
\Phi(B,b,s):= \zeta(s) \left( \sum_{k=1}^{\infty} (-1)^{k-1} \sum_{\substack{1 \le d_1 < \cdots < d_k \\b \mid C(B;\bm{d})}} \frac{\gcd(\bm{d})}{ \mathrm{lcm}(\bm{d})^{s}} \right).
\]

\subsection{Evaluation of the series $\Phi(B,b,s)$}
Now let us evaluate $\Phi(B,b,s)$. We first follow the steps we have done so far in the reverse order and use \cref{lem:coset.intersection} to obtain
\begin{align}
 \Phi(B,b,s)
 %:= & \zeta(s) \left( \sum_{k=1}^{\infty} (-1)^{k-1} \sum_{\substack{1 \le d_1 < \cdots < d_k \\b \mid C(B;\bm{d})}} \frac{\gcd(d_1,\ldots,d_k)}{ \mathrm{lcm}(d_1,\ldots,d_k)^{s}} \right)  \\
=&  \sum_{k=1}^{\infty} (-1)^{k-1} \sum_{\substack{1 \le d_1 < \cdots < d_k \\b \mid C(B;\bm{d})}} \gcd(\bm{d}) \sum_{\mathrm{lcm}(\bm{d}) \mid m} m^{-s} \nonumber \\ 
=& \sum_{m=1}^{\infty} m^{-s} \sum_{k=1}^{\infty}(-1)^{k-1} \sum_{\substack{1 \le d_1 < \cdots < d_k, \, d_i \mid m \\b \mid C(B;\bm{d})}} \gcd(\bm{d}) \nonumber \\ 
=&  \sum_{m=1}^{\infty} m^{-s} \sum_{k=1}^{\infty}(-1)^{k-1} \sum_{\substack{1 \le d_1 < \cdots < d_k, \, d_i \mid m \\b \mid C(B;\bm{d})}} \# \left( \bigcap_{i=1}^{k} \frac{1}{d_i} \Z / \Z \right). \label{eqn:Phi-step1}
\end{align}

Next, consider the following partition of the set of divisors $d>0$ of $m$:
\[
\{d \in \Z_{>0} : d \mid m \} = \bigsqcup_{\delta \mid \gcd(m, b^{\infty})}  \bigsqcup_{t \in U_b^2} \mathcal{D}_b(m, \delta, t),
\]
where $\mathcal{D}_b(m, \delta, t)$ is defined for $\delta \mid \gcd(m,b^{\infty})$, $t \in U_b^2$ by
\begin{equation}\label{eqn:defn-mathcal-D}
\mathcal{D}_b(m,\delta,t) = \{ d \mid m : \delta \mid d, (d/\delta)^2 \equiv t \Mod{b} \}.
\end{equation}
It will be shown in \cref{pro:intersection.cond2} that for $d_i\mid m$ and $b\mid B$, we have $b \mid C(B;\bm{d})$ if and only if all $d_i$'s are in the same $\mathcal{D}_b(m,\delta,t)$. This leads us to compute $\Phi(B,b,s)$ further from $\eqref{eqn:Phi-step1}$ as 
\begin{align}
\Phi(B,b,s)=& \sum_{m=1}^{\infty}m^{-s} \sum_{k=1}^{\infty}(-1)^{k-1} \sum_{\delta \mid \gcd(m,b^{\infty})} \sum_{t \in  U_b^2} \sum_{\substack{1 \le d_1 < \cdots< d_k \\ d_i \in \mathcal{D}_b(m,\delta,t)}} \# \left( \bigcap_{i=1}^{k} \frac{1}{d_i} \Z / \Z \right) \nonumber  \\
=&\sum_{m=1}^{\infty} m^{-s} \sum_{\delta\mid \gcd(m, b^{\infty})}\sum_{t \in  U_b^2} \sum_{k=1}^{\infty} (-1)^{k-1} \sum_{\substack{1 \le d_1 < \cdots< d_k \\ d_i \in \mathcal{D}_b(m,\delta,t)}} \# \left( \bigcap_{i=1}^{k} \frac{1}{d_i} \Z / \Z \right) \nonumber \\
=& \sum_{m=1}^{\infty} m^{-s} \sum_{\delta \mid \gcd(m,b^{\infty})} \sum_{t \in U_b^2} \# \left( \bigcup_{d \in \mathcal{D}_b(m,\delta,t)} \frac{1}{d} \Z / \Z \right). \label{eqn:Phi-step2}
\end{align}
In the last step, we used inclusion-exclusion principle. Writing $m=m_bm_{co}$ with $m_b\mid b^\infty$ and $\gcd(m_{co},b)=1$ and applying \cref{pro:unioncount.delta}, we may verify $\Phi(B,b,s)$ further from \eqref{eqn:Phi-step2} as
\begin{align}
\Phi(B,b,s)&=\sum_{m_b \mid b^{\infty}} \sum_{\gcd(m_{co},b)=1} m_b^{-s} m_{co}^{-s} \sum_{\delta \mid m_b} \sum_{t \in  U_b^2} \delta \cdot \# \left( \bigcup_{\substack{d \mid m_{co}\\ d^2 \equiv t \Mod{b}}} \frac{1}{d} \Z / \Z  \right) \nonumber  \\
&= \left( \sum_{m_b \mid b^{\infty}} m_b^{-s} \sum_{\delta \mid m_b} \delta \right) \left( \sum_{t \in  U_b^2} \sum_{\gcd(m_{co},b)=1} m_{co}^{-s} \cdot \# \left( \bigcup_{\substack{d \mid m_{co}\\ \left(\frac{m_{co}}{d}\right)^2 \equiv t \Mod{b}}} \frac{d}{m_{co}} \Z / \Z  \right) \right). \label{eqn:Phi-step3}
\end{align}

For $t\in U_b^2$, let us define
\begin{equation}\label{eqn:defn-Psi}
\Psi(b,t,s) = \sum_{\gcd(m,b)=1} m^{-s} \cdot \# \left( \bigcup_{\substack{d \mid m\\ d^2 \equiv t \Mod{b}}} \frac{d}{m}\Z/\Z \right).
\end{equation}
Note that for $m$ coprime to $b$, $m^2t^*\in  U_b^2$ runs over all $ U_b^2$ as $t$ varies in $ U_b^2$, where $t^*$ denotes an inverse of $t$ modulo $b$. Thus one may deduce from \eqref{eqn:Phi-step3} that
\begin{equation}\label{eqn:Phi-final}
\Phi(B,b,s) = \left( \sum_{m_b \mid b^{\infty}} m_b^{-s} \sum_{\delta \mid m_b} \delta \right) \sum_{t \in  U_b^2} \Psi(b,m^2t^*,s)
=\left( \sum_{m_b \mid b^{\infty}} m_b^{-s} \sum_{\delta \mid m_b} \delta \right) \sum_{t \in  U_b^2} \Psi(b,t,s).
\end{equation}
Plugging \eqref{eqn:Phi-final} back into \eqref{eqn:slzeta=SumofPhi} and then applying \cref{prop:slzeta-sum-b_part}, we have
\begin{align}
    \slzeta{L}{s} &= \sum_{b \mid B} B^{-s} F(b,s) \left( \sum_{m_b \mid b^{\infty}} m_b^{-s} \sum_{\delta \mid m_b} \delta \right) \sum_{t \in  U_b^2 } \Psi (b,t,s) \nonumber \\
    &=\sum_{b \mid B} (B/b)^{-s} \left( \prod_{p \mid b} \zeta_p(s-1) \right) \sum_{t \in  U_b^2} \Psi(b,t,s). \label{eqn:slzeta=SumofPsi-final}
\end{align}
\subsection{Evaluation of $\Psi(b,t,s)$} 
Hence the computation of $ \slzeta{L}{s}$ boils down to the evaluation of $\Psi(b,t,s)$.

Suppose $\gcd(m,b)=1$ and let $t\in U_b^2$. For an integer $1 \le a \le m$, we have
\begin{align*}
\frac{a}{m}  + \Z \in \left( \bigcup_{\substack{d \mid m\\ d^2 \equiv t \Mod{b}}} \frac{d}{m}\Z/\Z \right) \quad & \Leftrightarrow \quad \exists d \mid m \text{ such that } d^2 \equiv t \Mod{b} \text{ and } d \mid a \\
& \Leftrightarrow \quad \exists d \mid \gcd(a,m) \text{ such that } d^2 \equiv t \Mod{b}.
\end{align*}
For $g,b\in\n$ and $t \in U_b^2$, define
\begin{equation}\label{eqn:defn-X}
X(b,t,g) = \begin{cases} 1 & \text{if } \exists d \mid g \text{ such that } d^2 \equiv t \Mod{b},
\\
0 & \text{otherwise}.
\end{cases}
\end{equation}
We observe that
\[
\frac{a}{m}  + \Z  \in \left( \bigcup_{\substack{d \mid m\\ d^2 \equiv t \Mod{b}}} \frac{d}{m}\Z/\Z \right) \quad \Leftrightarrow \quad X(b,t,\gcd(a,m))=1.
\]

For a given $g \mid m$, since $\#\{1 \le a \le m : \gcd(a,m)=g\}=\varphi(\frac{m}{g})$, where $\varphi(n)$ denotes the Euler's totient function, we have

\[
    \# \left( \bigcup_{\substack{d \mid m\\ d^2 \equiv t \Mod{b}}} \frac{d}{m}\Z/\Z \right) = \sum_{g \mid m} \varphi \left(\frac{m}{g} \right) X(b,t,g).
\]
Plugging this into \eqref{eqn:defn-Psi}, we have 
\begin{align*}
\Psi(b,t,s)&= \sum_{\gcd(m,b)=1} \sum_{g \mid m} \varphi \left(\frac{m}{g} \right) X(b,t,g)m^{-s}\\
&= \sum_{\gcd(g,b)=1} X(b,t,g) \sum_{\substack{\gcd(m,b)=1 \\ g \mid m}} \varphi \left( \frac{m}{g} \right) m^{-s} \\
&= \sum_{\gcd(g,b)=1} X(b,t,g) g^{-s} \frac{\zeta(s-1)}{\zeta(s)}\prod_{p \mid b} \frac{\zeta_p(s)}{\zeta_p(s-1)},
\end{align*}
where we use the identity $\sum_{n=1}^\infty \varphi(n)n^{-s}=\frac{\zeta(s-1)}{\zeta(s)}$ in the last line.
Thus noting that $X(b,t,g)=X(b,t,gf)$ for any $\gcd(g,n)=1$ and $f\mid b^\infty$, we have
\begin{align}
 \left( \prod_{p \mid b} \zeta_p(s-1) \right) \Psi(b,t,s) =& \frac{\zeta(s-1)}{\zeta(s)} \sum_{\gcd(g,b)=1} X(b,t,g) g^{-s}  \prod_{p \mid b } \zeta_p(s) \nonumber \\
=& \frac{\zeta(s-1)}{\zeta(s)} \sum_{\gcd(g,b)=1} X(b,t,g) g^{-s} \sum_{f \mid b^{\infty}}f^{-s} \nonumber \\
=&\frac{\zeta(s-1)}{\zeta(s)} \sum_{\gcd(g,b)=1}\sum_{f \mid b^{\infty}} X(b,t,gf) (gf)^{-s} = \frac{\zeta(s-1)}{\zeta(s)} \sum_{m=1}^{\infty} \frac{X(b,t,m)}{m^s}. \label{eqn:Psi-into-SumofX}
\end{align}
Since $X(b,t,m)=1$ if and only if $t \in \{d^2 \Mod {b} : d \mid m, \gcd(d,b)=1 \}$, we have
\begin{equation}\label{eqn:SumofX-over-all-t}
\sum_{t \in  U_b^2} \sum_{m=1}^{\infty} \frac{X(b,t,m)}{m^s}  = \sum_{m=1}^{\infty} \frac{\# \{d^2 \Mod{b} : d \mid m, \gcd(d,b)=1\}}{m^s}.
\end{equation}
Summing \eqref{eqn:Psi-into-SumofX} over all $t\in U_b^2$ and then applying \eqref{eqn:SumofX-over-all-t}, we obtain
\[
\left( \prod_{p \mid b} \zeta_p(s-1) \right) \sum_{t \in U_b^2}\Psi(b,t,s)=\frac{\zeta(s-1)}{\zeta(s)} \sum_{m=1}^{\infty}  \frac{ \# \{d^2 \Mod {b} : d \mid m,\, \gcd(d,b)=1 \} }{m^s}.
\]
Plugging this back into \eqref{eqn:slzeta=SumofPsi-final}, we prove \cref{thm:MainAnswer}. 

\section{Technical details}\label{sec:lemmas}
In this section, we prove technical details that allowed us to compute $\slzeta{L}{s}$ for sublattices of $L\subseteq\mathbb{H}$.

\begin{prop} \label{prop:intersection.cond1}
Let $d_1,\ldots,d_k$ be distinct positive divisors of $m$, $d = \gcd(\bm{d})$, and let $M(B;\bm{d})$ be as defined in \eqref{eqn:defn-M}. Then 
\[
\bigcap_{i=1}^{k} S_{m,d_i, A/B} \neq \emptyset \quad \text{ if and only if } \quad  M(B;\bm{d}) \mid m.
\]
\end{prop}

\begin{proof}
Recalling the definition of $M(B;\bm{d})$, it suffices to show that
\[
\bigcap_{i=1}^{k} S_{m,d_i, A/B}  \neq \emptyset \quad \text{if and only if} \quad  \frac{1}{B} \frac{m}{d} \left( \frac{d_i}{d_j} - \frac{d_j}{d_i} \right) \in \Z \text{ for any }1 \le i< j \le k.
\]
To prove the only if part, let $x\in\q$ be such that $x+\z \in \bigcap_{i=1}^{k} S_{m,d_i, A/B}$. Then for each $i$, we have
\[
x \in \frac{A}{B} \frac{m/d_i}{d_i} + \frac{1}{d_i} \Z \quad \Rightarrow \quad d_i x \in \frac{A}{B} \frac{m}{d_i} + \Z \quad \Rightarrow \quad d_i x - \frac{A}{B}  \frac{m}{d_i} \in \Z.
\]
Hence for any $1\le i < j \le k$, we have
\[
\frac{d_j}{d} \left(d_i x -  \frac{A}{B} \frac{m}{d_i} \right) -\frac{d_i}{d} \left(d_j x -  \frac{A}{B} \frac{m}{d_j} \right) = \frac{A}{B} \frac{m}{d}  \left(\frac{d_i}{d_j} - \frac{d_j}{d_i} \right) \in \Z.
\]
Meanwhile we always have
\[
\frac{m}{d}\left(\frac{d_i}{d_j} - \frac{d_j}{d_i} \right) = \frac{m}{d_j} \frac{d_i}{d} - \frac{m}{d_i} \frac{d_j}{d} \in \Z.
\]
Hence
\[
\frac{m}{d}\left(\frac{d_i}{d_j} - \frac{d_j}{d_i} \right) \in \Z \cap \frac{B}{A} \Z = B\Z,
\]
where the last expression follows since $\gcd(A,B)=1$.

To prove the if part, let $\alpha_i \in \Z$ be integers satisfying $\sum_{j=1}^{k} \alpha_j d_j = d$ and define
\[
x = \frac{A}{B} \sum_{j=1}^{k} \frac{m}{d} \frac{\alpha_j}{d_j}.
\]
Then for each $i$ we have $x \in \frac{A}{B}\frac{m/d_i}{d_i}+ d_i^{-1} \Z =S_{m,d_i, A/B}$ as
\begin{align*}
d_i x =  \frac{A}{B} \sum_{j=1}^{k}  \frac{m}{d} \frac{d_i}{d_j} \alpha_j &= \frac{A}{B} \sum_{j=1}^{k} \frac{m}{d} \frac{d_j}{d_i} \alpha_j + A \sum_{j=1}^{k} \frac{1}{B}\frac{m}{d}\left( \frac{d_i}{d_j}- \frac{d_j}{d_i}\right) \alpha_j \\
&= \frac{A}{B} \frac{m}{d_i}+ A \sum_{j=1}^{k} \frac{1}{B}\frac{m}{d}\left( \frac{d_i}{d_j}- \frac{d_j}{d_i}\right) \alpha_j  \in \frac{A}{B}\frac{m}{d_i} + \Z. \qedhere
\end{align*}
\end{proof}

The number $M(B;\bm{d})$ can be calculated as follows.

\begin{prop} \label{pro:M.valuation}
Let $p$ be a prime and let $d=\gcd(\bm{d})$. Then $\nu_p(M(B;\bm{d}) )$ is determined as follows:
\begin{enumerate}[label={\rm (\arabic*)}]
    \item If $\nu_p(d_i)$ are not all the same, then
    \[
    \nu_p(M(B;\bm{d})) =  \max_{1 \le i \le k} \nu_p(d_i)+\nu_p(B).
    \]
    \item If $\nu_p(d_i)$ are all the same, then 
    \[
    \nu_p(M(B;\bm{d})) = \max_{1 \le i \le k} \nu_p(d_i)+\nu_p(B) -\min \left(\nu_p(B), \min_{1 \le i<j\le k} \nu_p \left(\frac{d_i^2}{d^2}-\frac{d_j^2}{d^2} \right) \right).
    \]
\end{enumerate}

\end{prop}

\begin{proof}
We first note that
\[
\nu_p(M(B;\bm{d}) ) = \max\left( \max_{1 \le i \le k} \nu_p(d_i), \max_{1 \le i < j \le k} \nu_p\left(Bd \left(\frac{d_i}{d_j} - \frac{d_j}{d_i} \right)^{-1} \right) \right).
\]
(1) Let $a,b\in\z$ with $1 \le a, b \le k$ be indices such that $\nu_p(d_a) = \displaystyle\min_{1 \le i \le k} \nu_p(d_i)$ and $\nu_p(d_b) =\displaystyle\max_{1 \le i \le k} \nu_p(d_i)$. Since
$\nu_p\left(\frac{d_i}{d_j}\right) \ge \nu_p(d_a)-\nu_p(d_b)$ for any $i ,j$, we have $\nu_p\left(\frac{d_i}{d_j}-\frac{d_j}{d_i}\right)\ge \nu_p(d_a)-\nu_p(d_b)$. Thus
\begin{equation} \label{eq:temp1}
\nu_p\left( B d \left( \frac{d_i}{d_j}- \frac{d_j}{d_i} \right)^{-1} \right) \le \nu_p(B) + \nu_p(d) + \nu_p(d_b) - \nu_p (d_a) = \nu_p (B) + \max_{1 \le i \le k} \nu_p(d_i)
\end{equation}
as $\nu_p(d) = \nu_p(d_a)$. Meanwhile, when $(i,j)=(a,b)$, one may check that
\[
\nu_p\left(\frac{d_a}{d_b}-\frac{d_b}{d_a}\right) = \nu_p\left(\frac{d_a}{d_b}\right) =  \nu_p(d_a ) - \nu_p (d_b).
\]
Hence the equality is achieved in \eqref{eq:temp1}, and we have
\[
\max_{1 \le i < j \le k} \nu_p\left( B d \left( \frac{d_i}{d_j}- \frac{d_j}{d_i} \right)^{-1} \right) = \nu_p(B) + \max_{1 \le i \le k} \nu_p(d_i).
\]
Thus
\[
\nu_p(M(B; \bm{d})) = \max\left(\max_{1 \le i \le k} \nu_p(d_i),\nu_p(B) + \max_{1 \le i \le k} \nu_p(d_i) \right) = \nu_p(B) + \max_{1 \le i \le k} \nu_p(d_i).
\]
(2) When $\nu_p(d_i)=\nu_p(d)$ for all $i$, we have
\[
\nu_p\left( \frac{d_i}{d_j} - \frac{d_j}{d_i} \right) = \nu_p(d_i^2 - d_j^2 ) - \nu_p(d_i d_j) = \nu_p(d_i^2 - d_j^2) - \nu_p(d^2) = \nu_p \left( \frac{d_i^2}{d^2} - \frac{d_j^2}{d^2} \right).
\]
Hence
\[
\max_{1 \le i < j \le k} \nu_p \left( Bd \left(\frac{d_i}{d_j} - \frac{d_j}{d_i} \right)^{-1} \right) = \nu_p(B) + \nu_p(d) - \min_{1 \le i < j \le k} \nu_p \left( \frac{d_i^2}{d^2} - \frac{d_j^2}{d^2} \right).
\]
Since $\displaystyle\max_{1 \le i \le k} \nu_p(d_i) = \nu_p(d)$, we have
\begin{align*}
\nu_p(M(B;\bm{d})) &= \max\left(\nu_p(d) , \nu_p(d) + \nu_p(B) - \min_{1 \le i < j \le k} \nu_p \left( \frac{d_i^2}{d^2} - \frac{d_j^2}{d^2} \right) \right) \\&= \nu_p(d) + \nu_p(B) + \max \left(-\nu_p(B),  -\min_{1 \le i < j \le k} \nu_p \left( \frac{d_i^2}{d^2} - \frac{d_j^2}{d^2} \right) \right)\\
&= \max_{1 \le i \le k} \nu_p(d_i) + \nu_p(B) - \min\left(\nu_p(B),  \min_{1 \le i < j \le k} \nu_p \left( \frac{d_i^2}{d^2} - \frac{d_j^2}{d^2} \right) \right).\qedhere
\end{align*}
\end{proof}

\begin{prop}\label{pro:intersection.cond2}
 For $m,B\in \n$, let $d_1,\ldots,d_k$ be distinct divisors of $m$, $C(B;\bm{d})$ be as defined in \eqref{eqn:defn-C}, and $d = \gcd(\bm{d})$. Let $b$ be a divisor of $B$. For $\delta \mid \gcd(m,b^{\infty})$, $t \in U_b^2$, let
\[
\mathcal{D}_b(m,\delta,t) = \{ g \mid m : \delta \mid g, (g/\delta)^2 \equiv t \Mod{b} \},
\] 
 as defined in \eqref{eqn:defn-mathcal-D}. Then the following conditions are equivalent:
\begin{enumerate}[label={\rm (\arabic*)}]
    \item $b \mid C(B;\bm{d})$.
    \item There exists $t \in U_b^2$ such that $\frac{d_i^2}{d^2} \equiv t \Mod{b}$ for all $ 1 \le i \le k$.
    \item $d_1,\ldots,d_k$ are all in the same $\mathcal{D}_b(m,\delta,t)$.
\end{enumerate}
\end{prop}

\begin{proof}
($(1)\Rightarrow(2)$) Recall $C(B;\bm{d}) = \frac{B \cdot \mathrm{lcm}(\bm{d})}{M(B;\bm{d})}$. Suppose that $b \mid C(B;\bm{d})$, that is, $\nu_p(b) \le \nu_p(C(B;\bm{d}))$ for any prime $p$. By Proposition \ref{pro:M.valuation} we have
\[
\nu_p(C(B;\bm{d})) = \begin{cases}
    \min \left(\nu_p(B), \displaystyle\min_{1 \le i  < j \le k} \nu_p \left(\frac{d_i^2}{d^2}-\frac{d_j^2}{d^2} \right) \right) & \text{if } \nu_p(B)>0 \text{ and } \nu_p(d_1) = \cdots = \nu_p(d_k),\\
    0 & \text{otherwise}.
\end{cases}
\]
Hence if $\nu_p(b)>0$, then $\nu_p(C(B;\bm{d}))>0$, and hence $\nu_p(d_i)=\nu_p(d)$ for all $i$ and for all $1 \le i < j \le k$,
\[
\nu_p\left( \dfrac{d_i^2}{d^2} - \dfrac{d_j^2}{d^2} \right) \ge \nu_p(b).
\]
Thus, we have $\gcd\left( d_i/d, b \right) = 1$ for all $i$ and $b \mid \frac{d_i^2}{d^2} - \frac{d_j^2}{d^2}$ for all $1 \le i < j \le k$. 

($(2)\Rightarrow(3)$) Suppose there exists $t \in U_b^2$ such that $(d_i/d)^2 \equiv t \Mod{b}$. As $\gcd(d_i/d,b)=1$, we have $\gcd(d_i, b^{\infty}) = \gcd(d, b^{\infty})$. Letting $\delta = \gcd(d, b^{\infty})$, we have $\delta \mid d$ and
\[
\frac{d_i^2}{\delta^2} \equiv \frac{d_i^2}{d^2} \frac{\delta^2}{d^2} \equiv t \frac{\delta^2}{d^2} \Mod{b}.
\]
Thus $d_i \in \mathcal{D}_b(m,\delta,t(\delta/d)^2)$ for all $i$.

($(3)\Rightarrow(1)$) Suppose $d_1, \ldots, d_k \in \mathcal{D}_b(m,\delta,t)$. Since $\gcd(d_i/\delta, b)=1$, for any prime $p\mid b$, we have $\nu_p(d_i) = \nu_p(\delta)=\nu_p(d)$ for all $i$. For any $1 \le i < j \le k$, since $b$ divides $(d_i/\delta)^2 - (d_j/\delta)^2$, we have
\[
\nu_p(d_i^2 - d_j^2) \ge \nu_p(\delta^2) + \nu_p(b) = \nu_p(d^2)+\nu_p(b).
\]
Thus by Proposition \ref{pro:M.valuation} and since $b \mid B$, we have for any prime $p \mid b$ that
\[
\nu_p(C(B;{\bm d})) = \min\left( \nu_p(B), \min_{1 \le i < j \le k} \nu_p \left( \frac{d_i^2}{d^2} - \frac{d_j^2}{d^2} \right) \right) \ge \min(\nu_p(B), \nu_p(b) ) \ge \nu_p(b). \qedhere
\]
\end{proof}

\begin{prop} \label{pro:unioncount.delta}
Let $m,b\in\n$, $t\in U_b^2$, and write $m=m_bm_{co}$ with $m_b\mid b^\infty$ and $\gcd(m_{co},b)=1$. Then 
\[
\# \left( \bigcup_{d \in \mathcal{D}_b(m,\delta,t)} d^{-1} \Z / \Z \right) = \delta \cdot \# \left( \bigcup_{\substack{d \mid m_{co}\\ d^2 \equiv t \Mod{b}}} d^{-1} \Z / \Z  \right).
\]
\end{prop}

\begin{proof}
We show that
\[
\mathcal{D}_b(m,\delta ,t)  = \{ \delta d : d \mid m_{co}, d^2 \equiv t \Mod{b} \}.
\]
Suppose $u \in \mathcal{D}_b(m,\delta, t)$. Then $\delta \mid u$ and $(u/\delta)^2 \equiv t \Mod{b}$. Denote $d = u/\delta$ then $d^2 \equiv t \Mod{b}$. As $d \mid m$ and $\gcd(d,b)=1$, it follows that $d \mid m_{co}$. Conversely if $d \mid m_{co}$ and $d^2 \equiv t \Mod{b}$, then $\delta d \in \mathcal{D}_b(m,\delta,t)$ is obvious. We note that
\[
\kappa : 
\bigcup_{\substack{d \mid m_{co}\\ d^2 \equiv t \Mod{b}}} (\delta d)^{-1}  \Z / \Z \rightarrow \bigcup_{\substack{d \mid m_{co}\\ d^2 \equiv t \Mod{b}}} d^{-1} \Z / \Z  , \quad  \kappa(x +\Z ) =  \delta x +\Z
\]
is a $\delta$-to-$1$ surjective map. Thus
\[
\# \left( \bigcup_{d \in \mathcal{D}_b(m,\delta,t)} d^{-1} \Z / \Z \right) =
\# \left( \bigcup_{\substack{d \mid m_{co}\\ d^2 \equiv t \Mod{b}}} (\delta d)^{-1} \Z / \Z  \right)=
\delta \cdot \# \left( \bigcup_{\substack{d \mid m_{co}\\ d^2 \equiv t \Mod{b}}} d^{-1} \Z / \Z  \right). \qedhere
\]
\end{proof}

\begin{prop}\label{prop:slzeta-sum-b_part}
For $b\in\n$, let $F(b,s)$ be as defined in \eqref{eqn:Ns-Mobius}. We have
\[
F(b,s) \left( \sum_{n \mid b^{\infty}} n^{-s} \sum_{\delta \mid n} \delta \right)  = b^{s} \cdot \prod_{p \mid b} \zeta_p(s-1).
\]
\end{prop}

\begin{proof}
Let $\sigma_1(m):=\sum_{d\mid m}d$. Then $\sigma_1$ is multiplicative and it is well known that
\[
\sum_{m=1}^{\infty} \frac{\sigma_1(m)}{m^s} = \zeta(s)\zeta(s-1).
\]
Hence 
\[
\sum_{n \mid b^{\infty}} n^{-s} \sum_{\delta \mid n} \delta = \sum_{n \mid b^{\infty}} \frac{\sigma_1(n)}{n^{s}} = \prod_{p\mid b} \zeta_p(s)\zeta_p(s-1).
\]
Meanwhile since
\[
F(b,s) = \sum_{d \mid b} d^{s} \mu\left(\frac{b}{d}\right) = b^s \sum_{d \mid b } \left(\frac{b}{d}\right)^{-s} \mu\left(\frac{b}{d}\right) = b^s \prod_{p \mid b}(1-p^{-s}),
\]
we have
\[
F(b,s)\left( \sum_{n \mid b^{\infty}} n^{-s} \sum_{\delta \mid n} \delta \right) =  b^s \prod_{p \mid b} (1-p^{-s})  \cdot\left( \prod_{p \mid b} \zeta_p(s) \zeta_p(s-1) \right) = b^s \prod_{p\mid b}\zeta_p(s-1). \qedhere
\]
\end{proof}

\section{Proof of Theorem \ref{thm:MainAnswer2}}\label{sec:comb}
In this section, we prove Theorem \ref{thm:MainAnswer2}. 
Recall that from \cref{thm:MainAnswer} and \eqref{eqn:SumofX-over-all-t}, we have
\begin{align*}
    \slzeta{L}{s}&=\frac{\zeta(s-1)}{\zeta(s)}\sum_{b \mid B} (B/b)^{-s} \sum_{m=1}^{\infty}  \frac{ \# \{d^2 \Mod {b} : d \mid m,\, \gcd(d,b)=1 \} }{m^s}\\
    &=\frac{\zeta(s-1)}{\zeta(s)} \sum_{b \mid B} (B/b)^{-s} \sum_{t \in  U_b^2} \sum_{m=1}^{\infty} \frac{X(b,t,m)}{m^s}.
\end{align*}
Since the quantity $X(b,t,m)$ defined in \eqref{eqn:defn-X} concerns the set $\{ d^2 \Mod{b} : d \mid m \}$, the prime divisors of $m$ whose squares are equivalent modulo $b$ seem to play the same role. Hence it is natural to group all primes into squares mod $b$ as
\[
\mathcal{P}_t^{(b)} = \{ p \text{ prime }: p^2 \equiv t \Mod{b} \} \quad \text{and} \quad \mathcal{P}_{0}^{(b)} = \{ p \text{ prime }: p \mid b \}
\]
for each $t \in  U_b^2$ and consider
\[
\Omega(m) = \left( \sum_{p \in \mathcal{P}_t^{(b)}} \nu_p(m)
\right)_{t \in U_b^2 }.
\]
We will show that $X(b,t,m)$ only depends on $\Omega(m)$, allowing us to collect the terms in the series $\sum_m X(b,t,m) m^{-s}$ in terms of the values of $\Omega(m)$. 

To demonstrate the proof, let us define
\[
W(=W(b)):= \{ {\bm w} = (w_{u})_{u \in U_b^2} : w_u \in \Z_{\ge 0} \} \subseteq{\mathbb{R}^{| U_b^2|}}.
\]
For $t \in U_b^2$, define
\[
\mathcal{S}_t := \left\{ {\bm w} \in W : \prod_{u \in U_b^2} u^{w_u} \equiv t \Mod{b} \right\},
\]
and let
\begin{equation*}
    \mathcal{T}_t:=\mathcal{S}_t+W\subseteq W \quad \text{and} \quad \mathcal{T}_t^{c}:=W \setminus \mathcal{T}_t.
\end{equation*}
Moreover, for ${\bm w} \in W$, let us define
\[
c_b({\bm w}) = \# \left\{ t \in  U_b^2 : {\bm w} \notin \mathcal{T}_t \right\}.
\]
%Let us also define
%\[\mathcal{P}_t^{(b)} = \{ p \text{ prime }: p^2 \equiv t \Mod{b} \} \quad \text{and} \quad \mathcal{P}_{0}^{(b)} = \{ p \text{ prime }: p \mid b \}.\]

For an integer $k \ge 0$ and an infinite list $X = (x_1, x_2, \cdots)$, denote by $\mathrm{Sym}_k(X) = \mathrm{Sym}_k(x_1, x_2, \cdots)$ the $k$-th complete homogeneous symmetric polynomial
\[
\mathrm{Sym}_k(x_1, x_2, \cdots ) = \sum_{1 \le i_1 \le \cdots \le i_k} x_{i_1} x_{i_2} \cdots x_{i_k}.
%[y^k]\prod_{i=1}^{\infty} \frac{1}{1-y x_i} =
\]

\begin{prop}
Let $X(b,t,m)$ be as defined in \eqref{eqn:defn-X}. We have
\begin{align}
\sum_{m=1}^{\infty} \frac{X(b,t,m)}{m^s} 
%&= \sum_{{\bm w} \in \mathcal{T}_t} \prod_{p \in \mathcal{P}_0^{(b)}} (1-p^{-s})^{-1} \prod_{u \in U_b^2 }  \mathrm{Sym}_{w_u}(\{p^{-s} : p \in \mathcal{P}_u^{(b)}\}) \label{eqn:X-series-into-Sym1}  \\
&=\zeta(s)-\sum_{{\bm w} \in \mathcal{T}_t^c}  \prod_{p \in \mathcal{P}_0^{(b)}} (1-p^{-s})^{-1}  \prod_{u \in U_b^2 }  \mathrm{Sym}_{w_u}(\{p^{-s} : p \in \mathcal{P}_u^{(b)}\}) \label{eqn:X-series-into-Sym2}
\end{align}

\end{prop}

\begin{proof}
Define a function $\Omega : \Z_{> 0} \rightarrow W$ by
\[
\Omega(m) = \left( \sum_{p \in \mathcal{P}_u^{(b)}} \nu_p(m)
\right)_{u \in U_b^2 } \in W.
\]

We first show the following equivalent condition:
\begin{equation}\label{eqn:X-Omega-equiv}
X(b,t,m)=1 \quad \Leftrightarrow \quad \Omega(m) \in \mathcal{T}_t.
\end{equation}
If $X(b,t,m)=1$ then there exists $d \mid m$ such that $d^2 \equiv t\Mod{b}$. Then we have
\[
    \prod_{u \in U_b^2} \prod_{p \in \mathcal{P}_u^{(b)}}(p^2)^{\nu_p(d)} \equiv t \Mod{b} \quad \Rightarrow \quad \prod_{u \in  U_b^2} \prod_{p \in \mathcal{P}_u^{(b)}} u^{\nu_p(d)} \equiv t \Mod{b} \quad \Rightarrow \quad \Omega(d)  \in \mathcal{S}_t.
\]
Thus $\Omega(m) = \Omega(d) + \Omega(m/d) \in \mathcal{S}_t+W=\mathcal{T}_t$.

Conversely, suppose that $\Omega(m) \in \mathcal{T}_t$. Let ${\bm w} = (w_u)_{u\in U_b^2} \in \mathcal{S}_t$, ${\bm z} \in W$ be such that $\Omega(m) = {\bm w} + {\bm z}$. For each $u \in U_b^2$, let $p_{u,1}, \ldots, p_{u,i_u}$ be all distinct prime divisors of $m$ in $\mathcal{P}_u^{(b)}$. Then
\[
0 \le w_u \le (\Omega(m))_u = \sum_{j=1}^{i_u}\nu_{p_{u,j}}(m).
\]
Let us pick integers $f_{u,j}$ satisfying $0 \le f_{u,j} \le \nu_{p_{u,j}}(m)$ and $\sum_{j=1}^{i_u} f_{u,j}=w_u$, and define
\[
d = \prod_{u \in  U_b^2} \prod_{j=1}^{i_u}p_{u,j}^{f_{u,j}}.
\]
Then from the construction, we have $d \mid m$, $\gcd(d,b)=1$ and $\Omega(d) = \bm{w} \in \mathcal{S}_t$. Hence $d^2 \equiv t \Mod{b}$.

Now \eqref{eqn:X-Omega-equiv} implies that
\[
\sum_{m=1}^{\infty} \frac{X(b,t,m)}{m^s} = \sum_{{\bm w} \in \mathcal{T}_t} \left( \sum_{\Omega(m)={\bm w}} m^{-s} \right) = \zeta(s) - \sum_{{\bm w} \in \mathcal{T}_t^c} \left( \sum_{\Omega(m)={\bm w}} m^{-s} \right).
\]
By considering the prime factorization of $m$ with $\Omega(m)=\bm{w}$ and matching the coefficients, one gets
\[
\sum_{\Omega(m) = {\bm w}} m^{-s} = \prod_{p \in \mathcal{P}_0^{(b)}} (1-p^{-s})^{-1} \prod_{u \in U_b^2 }  \mathrm{Sym}_{w_u}(\{p^{-s} : p \in \mathcal{P}_u^{(b)}\}),
\]
which proves the proposition.
\end{proof}
We finally prove Theorem \ref{thm:MainAnswer2}. Let $L$ be a sublattice of $\mathbb{H}$ and let $B\in\n$ be such that $B\z=[\mathbb{H}:L](\mathfrak{n}L)^{-1}$. Recall that in Theorem \ref{thm:MainAnswer}, we have 
\begin{equation}\label{eqn:recall-mainthm1}
\slzeta{L}{s}=\sum_{b \mid B} (B/b)^{-s} \frac{\zeta(s-1)}{\zeta(s)} \sum_{m=1}^{\infty}  \frac{ \# \{d^2 \Mod {b} : d \mid m,\, \gcd(d,b)=1 \} }{m^s}.
\end{equation}
By \eqref{eqn:SumofX-over-all-t} and \eqref{eqn:X-series-into-Sym2}, we have
\begin{align*}
\sum_{m=1}^{\infty} &\frac{\# \{d^2 \Mod{b} : d \mid m, \gcd(d,b)=1\}}{m^s}=\sum_{t \in  U_b^2} \sum_{m=1}^{\infty} \frac{X(b,t,m)}{m^s}\\
&=\sum_{t \in U_b^2} \left( \zeta(s)-\sum_{{\bm w} \in \mathcal{T}_t^c}  \prod_{p \in \mathcal{P}_0^{(b)}} (1-p^{-s})^{-1}  \prod_{u \in U_b^2 }  \mathrm{Sym}_{w_u}(\{p^{-s} : p \in \mathcal{P}_u^{(b)}\})  \right)\\
&=\zeta(s)|U_{b}^{2}|- \sum_{{\bm w} \in W}c_{b}({\bm w})  \prod_{p \in \mathcal{P}_0^{(b)}} (1-p^{-s})^{-1}  \prod_{u \in U_b^2 }  \mathrm{Sym}_{w_u}(\{p^{-s} : p \in \mathcal{P}_u^{(b)}\}) \\
&=\zeta(s)|U_{b}^{2}| - \zeta(s) \sum_{{\bm w} \in W} c_{b}({\bm w})   \prod_{u \in U_b^2 }  \frac{\mathrm{Sym}_{w_u}(\{p^{-s} : p \in \mathcal{P}_u^{(b)}\})}{\prod_{p \in \mathcal{P}_u^{(b)}}(1-p^{-s})^{-1}}.     
\end{align*}
Plugging this into \eqref{eqn:recall-mainthm1}, we obtain the desired formula:
\[
\slzeta{L}{s} =\zeta(s-1)\sum_{b \mid B} (B/b)^{-s}\left(  \# U_b^2  - \sum_{{\bm w}\in W} c_b({\bm w}) \prod_{u \in  U_b^2  } \frac{\mathrm{Sym}_{w_u}(\{p^{-s} : p \in \mathcal{P}_u^{(b)} \} )}{\prod_{p \in \mathcal{P}_u^{(b)}}(1-p^{-s})^{-1}} \right).
\]
This proves Theorem \ref{thm:MainAnswer2}.

\section{Explicit examples}\label{sec:explicit.exm}
In this section, we demonstrate how one can explicitly compute $\slzeta{L}{s}$ for various cases. Throughout this section, let $L$ be a sublattice of $\mathbb{H}$ and let $B\in\n$ be such that $B\z=[\mathbb{H}:L](\mathfrak{n}L)^{-1}$. 

Let us first define
\begin{equation}\label{eqn:defn-Y,H}
Y_{u,w}(s) := \frac{\mathrm{Sym}_{w}(\{p^{-s} : p \in \mathcal{P}_u^{(b)}\})}{\prod_{p \in \mathcal{P}_u^{(b)} }(1-p^{-s})^{-1}} \quad \text{and} \quad
\mathcal{H}_b(s) := |U_b^2| - \prod_{{\bm w} \in W} c_b({\bm w}) \prod_{u \in  U_b^2}Y_{u,w}.
%\frac{\mathrm{Sym}_{w_u}(\{p^{-s} : p \in \mathcal{P}_u^{(b)}\})}{\prod_{p \in \mathcal{P}_u^{(b)} }(1-p^{-s})^{-1}}.
\end{equation}
Note that with these notation, Theorem \ref{thm:MainAnswer2} becomes
\begin{equation}\label{eqn:MainAnswer3}
\slzeta{L}{s} = \zeta(s-1) \sum_{b \mid B} (B/b)^{-s} \mathcal{H}_b(s).
\end{equation}
We further observe that
\begin{equation}\label{eqn:sum-all-Y=1}
\sum_{w=0}^{\infty} Y_{u,w}(s)=\sum_{w=0}^{\infty} \frac{\mathrm{Sym}_{w}( \{p^{-s} : p \in \mathcal{P}_u^{(b)}\})}{\prod_{p \in \mathcal{P}_u^{(b)}} (1-p^{-s})^{-1}} = \frac{\prod_{p \in \mathcal{P}_u^{(b)}} (1+p^{-s} + p^{-2s}+ \cdots)}{\prod_{p \in \mathcal{P}_u^{(b)}} (1-p^{-s})^{-1}} = 1.
\end{equation}

It remains to compute the integers $c_b(\bm{w})$, where an interesting combinatorial problem related to the sets $\mathcal{S}_t$ and $\mathcal{T}_t$ arise. In principle, there is a general algorithm to compute $c_b(\bm{w})$ in finite steps.

\begin{prop}
For any $\bm{w}=(w_u)\in W$, we have
\[
|U_b^2| - c_b(\bm{w}) = \#\left\{ \prod_{u \in U_b^{2}} u^{x_u} \Mod{b} : 0 \le x_u \le \min\left(w_u, |U_b^2 |\right) \right\}.
\]
\end{prop}

\begin{proof}
From the definition of $c_b(\bm{w})$ we have 
\[
|U_b^2| - c_b(\bm{w}) = \# \left\{ t \in U_b^2 : \bm{w} \in \mathcal{T}_t \right\}.
\]
Note that $\bm{w} \in \mathcal{T}_t$ if and only if there exists $\bm{x} \in \mathcal{S}_t$ such that $0 \le x_u \le t_u$. Since $\bm{x} \in \mathcal{S}_t$ if and only if $\prod_u u^{x_u} \equiv t \Mod{b}$, we have
\[
\bm{w} \in \mathcal{T}_t \quad \Leftrightarrow \quad t \in \left\{ \prod_{u \in U_b^{2}} u^{x_u} \Mod{b} : 0 \le x_u \le w_u \right\}.
\]
Since the set $\{ u^x \Mod{b} : 0 \le x_u \le w_u \}$ does not change if $w_u \ge |U_b^2|$, we have
\[
\left\{ \prod_{u \in U_b^{2}} u^{x_u} \Mod{b} : 0 \le x_u \le w_u \right\}=\left\{ \prod_{u \in U_b^{2}} u^{x_u} \Mod{b} : 0 \le x_u \le \min\left(w_u, |U_b^2|\right) \right\}.\qedhere
\]
\end{proof}
However, studying an effective way to compute $c_b(\bm{w})$ seems to be a challenging problem. Thus we list out simpler cases where $|U_b^{2}| \le 3$. It is well-known that for odd prime $p$ and $k \ge 1$
\[
(\Z/p^k \Z)^{\times}  \simeq \Z / p^{k-1}(p-1)\z \quad \text{and} \quad (\Z/2^k\Z)^{\times} \simeq \begin{cases} \{1\} & k \le 2, \\ \Z/2 \times \Z/2^{k-2} & k \ge 3. \end{cases}
\]
From this, one can deduce that
\begin{align*}
U_b^2 \simeq \{1\} \quad & \Leftrightarrow \quad b \in \{1,2,3,4,6,8,12,24\}, \\
U_b^2 \simeq \Z/2\Z \quad &\Leftrightarrow \quad b \in \{5,10,15,16,20,30,40,48,60,120\}, \\
U_b^2 \simeq \Z/3\Z \quad &\Leftrightarrow \quad b \in \{7, 9, 14, 18, 21, 28, 36, 42, 56, 72, 84, 168 \}.
\end{align*}

\begin{prop}\label{pro:computation.Hb}
Let $\mathcal{H}_b(s)$ be as defined in \eqref{eqn:defn-Y,H}.
\begin{enumerate}[leftmargin=*, label={\rm (\arabic*)}]
    \item If $|U_b^2|=1$, then $\mathcal{H}_b(s)=1$.
    \item If $|U_b^2|=2$, then
    \[
    \mathcal{H}_b(s) = 2 - \prod_{\substack{ \gcd(p,b)=1 \\ p^2 \not \equiv 1 \Mod{b}}} (1-p^{-s}).
    \]
    \item If $|U_b^2|=3$, then
    \[
\mathcal{H}_b(s) = 3 - 2 \prod_{\substack{\gcd(p,b)=1\\ p^2 \not \equiv 1 \Mod{b}}}(1-p^{-s} )- \left(\sum_{\substack{\gcd(p,b)=1\\p^2 \not \equiv 1 \Mod{b}} }p^{-s} \right) \prod_{\substack{\gcd(p,b)=1 \\p^2 \not \equiv 1 \Mod{b}}}(1-p^{-s} ).
\]
\end{enumerate}

\end{prop}

\begin{proof}
\noindent (1) Write $U_b^{2} = \{1\}$. Since $\mathcal{S}_1=\Z_{\ge 0}$, we have $\mathcal{T}_1^c = \emptyset$ and $c_b({\bm w})=0$ for all ${\bm w}$. Thus $\mathcal{H}_b(s)=1$.

(2) Write $U_b^2 = \{1,v\}$. Then we have
\[
\mathcal{S}_1 = \{ (w_1, w_v) \in \Z_{\ge 0}^2 : w_v \equiv 0 \Mod{2} \}, \quad \mathcal{S}_v = \{ (w_1, w_v) \in \Z_{\ge 0}^2 : w_v \equiv 1 \Mod{2} \}.
\]
Hence one can check that
\[
\mathcal{T}_1 = \Z_{\ge 0}^2, \quad \mathcal{T}_v = \{ (w_1, w_v) \in \Z_{\ge 0}^2 : w_v \ge 1 \}.
\]
Therefore we may observe that $c_b(w_1, w_v)=1$ if and only if $w_v=0$, and hence by \eqref{eqn:sum-all-Y=1} we have
\[
\mathcal{H}_b(s) = 2 - \sum_{w_1=0}^{\infty} Y_{1,w_1}(s) Y_{v,0}(s) = 2 - Y_{v,0}(s) = 2 - \prod_{p \in \mathcal{P}_v^{(b)}}(1-p^{-s}).
\]
Since $p \in \mathcal{P}_v^{(b)}$ if and only if $\gcd(p,b)=1$ and $p^2 \not \equiv 1 \Mod{b}$, we obtain the formula.

(3) Write $U_b^{2} = \{1, v, v^2\}$. Then for $k\in\{0,1,2\}$ we have
\[
\mathcal{S}_{v^{k}} = \{(w_1, w_{v}, w_{v^2}) \in \Z_{\ge 0}^3 : w_{v} + 2 w_{v^2} \equiv k \Mod{3} \}.
\]
One may calculate that
\[
\mathcal{T}_1 = \Z_{\ge 0}^3, \quad \mathcal{T}_{v} = \{(w_1, w_v, w_{v^2}) : w_v \ge 1 \text{ or } w_{v^2} \ge 2\}, \quad \mathcal{T}_{v^2} = \{(w_1, w_v, w_{v^2}) : w_v \ge 2 \text{ or } w_{v^2} \ge 1\}.
\]
Hence we have
\[
c_b(w_1,0,0)=2,\, c_b(w_1,1,0)= c_b(w_1,0,1)=1, \text{ and $c_b(\bm{w})=0$ for all other $\bm{w}$.}
\]
 Hence this, together with \eqref{eqn:sum-all-Y=1}, implies that 
\begin{align*}
\mathcal{H}_b(s) &= 3 -2 \sum_{w=0}^{\infty}  Y_{1,w}Y_{v,0}Y_{v^2,0} - \sum_{w=0}^{\infty} (Y_{1,w} Y_{v,0}Y_{v^2,1}+Y_{1,w}Y_{v,1}Y_{v^2,0}) \\
&=3 - 2 Y_{v,0} Y_{v^2,0} - Y_{v,0}Y_{v^2,1} - Y_{v,1}Y_{v^2,0} \\
&= 3 - 2Y_{v,0} Y_{v^2,0} - Y_{v,0}Y_{v^2,0} \left( \frac{ Y_{v,1}}{Y_{v,0}}+\frac{Y_{v^2,1}}{Y_{v^2,0}} \right).
\end{align*}
On the other hand, we have
\[
Y_{v,0}Y_{v^2,0} = \prod_{\substack{\gcd(p,b)=1\\ p^2 \not  \equiv1  \Mod{b}}} (1-p^{-s})\quad \text{and} \quad \frac{ Y_{v^k,1}}{Y_{v^k,0}} = \sum_{p \in \mathcal{P}_{v^k}^{(b)}} p^{-s}.
\]
Therefore, we conclude that
\[
\mathcal{H}_b(s) = 3 - 2 \prod_{\substack{\gcd(p,b)=1\\ p^2 \not \equiv 1 \Mod{b}}}(1-p^{-s} )- \left(\sum_{\substack{\gcd(p,b)=1\\p^2 \not \equiv 1 \Mod{b}} }p^{-s} \right) \prod_{\substack{\gcd(p,b)=1 \\p^2 \not \equiv 1 \Mod{b}}}(1-p^{-s} ). \qedhere
\]
\end{proof}

With \eqref{eqn:MainAnswer3} and Proposition \ref{pro:computation.Hb}, one can explicitly compute $\slzeta{L}{s}$ for certain sublattice $L$ of $\mathbb{H}$. Since $\slzeta{L}{s}$ is invariant of $B$, we record them according as $B$ the propositions below.
\begin{prop}
If $L$ is a sublattice of $\mathbb{H}$ with $B \in \{1,2,3,4,6,8,12,24\}$, then
\[
\slzeta{L}{s} = \zeta(s-1) \left(\sum_{b \mid B} b^{-s} \right).
\]
\end{prop}
\begin{proof}
    Note that for $B \in \{1,2,3,4,6,8,12,24\}$, its divisor $b|B$ is also in $\{1,2,3,4,6,8,12,24\}$. By Proposition \ref{pro:computation.Hb} $\mathcal{H}_b(s)=1$ for all such $b|B$, and we get
    \[
    \slzeta{L}{s} = \zeta(s-1) \sum_{b \mid B} (B/b)^{-s} \mathcal{H}_b(s)=\zeta(s-1) \sum_{b \mid B} (B/b)^{-s}. \qedhere
    \]
\end{proof}

\begin{prop}
    \begin{enumerate}[leftmargin=*, label={\rm (\arabic*)}]
        \item For $B=5$, we have
\[
\slzeta{L}{s} = \zeta(s-1) \left(2+5^{-s} - \prod_{p^2 \equiv 4 \Mod{5}}(1-p^{-s}) \right).
\]
\item For $B=10$, we have
\[
\slzeta{L}{s} = \zeta(s-1) \left( (1+2^{-s})(2+5^{-s})- (1+2^{-s}-4^{-s}) \prod_{p^2 \equiv 9 \Mod{10}}
(1-p^{-s})\right).
\]
\item For $B=15$, we have
\[
\slzeta{L}{s}= \zeta(s-1)\left( (1+3^{-s})(2+5^{-s}) - (1+3^{-s}-9^{-s}) \prod_{p^2 \equiv 4 \Mod{15}} (1-p^{-s})\right).
\]

\item For $B=20$, we have
\[
    \slzeta{L}{s} = \zeta(s-1)\left( (1+2^{-s}+4^{-s})(2+5^{-s}) -(1+2^{-s}+4^{-s}-8^{-s}) \prod_{p^2 \equiv 9 \Mod{10}}(1-p^{-s})\right).
\]
   \end{enumerate}
\end{prop}

\begin{proof}
    For $b=5$ we have
\[
\mathcal{H}_5(s) = 2 \zeta(s-1) - \zeta(s-1) \prod_{p^2 \equiv 4 \Mod{5}} (1-p^{-s}).
\]
For $b=10,15,20$, we need to exclude primes $p$ such that $p \mid b$ from the product inside, so we have
\[
\mathcal{H}_{10}(s) = \mathcal{H}_{20}(s)= 2 \zeta(s-1) - \zeta(s-1) \prod_{p^2 \equiv 9 \Mod{10}}(1-p^{-s})
\]
and
\[
\mathcal{H}_{15}(s) = 2 \zeta(s-1) - \zeta(s-1)  \prod_{p^2 \equiv 4 \Mod{15}}(1-p^{-s}). \qedhere
\]
\end{proof}
We provide several more formulae of $\slzeta{L}{s}$ for sublattices $L$ involving terms $\mathcal{H}_b(s)$ with $|U_b^2|=3$.
\begin{prop}
    \begin{enumerate}[leftmargin=*, label={\rm (\arabic*)}]
        \item For $B=7$, we have
\[
\slzeta{L}{s} = \zeta(s-1)\left((3 + 7^{-s}) - \left(2+\sum_{p^2 \equiv 2,4 \Mod{7}}p^{-s} \right) \prod_{p^2 \equiv 2,4 \Mod{7}}(1-p^{-s})\right).
\]
\item For $B=9$, we have
\[
    \slzeta{L}{s} = \zeta(s-1)\left((3+3^{-s}+9^{-s}) - \left(2+\sum_{p^2 \equiv 4,7 \Mod{9}}p^{-s} \right) \prod_{p^2 \equiv 4,7 \Mod{9}}(1-p^{-s}) \right).
\]

\item For $B=14$, we have
\begin{multline*}
    \slzeta{L}{s}=\zeta(s-1)\Bigg((1+2^{-s})(3+7^{-s}) - \prod_{p \equiv 9,11 \Mod{14}} (1-p^{-s})\\
     \times\left[2+2 \cdot 2^{-s}- 4^{-s}-8^{-s} + (1+2^{-s}-4^{-s})\sum_{p^2 \equiv 9,11 \Mod{14}} p^{-s} \right]\Bigg).
\end{multline*}

\item For $B=18$, we have
\begin{multline*}
\slzeta{L}{s} =\zeta(s-1) \Bigg((1+2^{-s})(3+3^{-s}+9^{-s}) - \prod_{p^2 \equiv 7,13 \Mod{18}}(1-p^{-s})\\
\times\left[2+2 \cdot 2^{-s} - 4^{-s} - 8^{-s} + (1+2^{-s}-4^{-s})  \sum_{p^2 \equiv 7,13 \Mod{18}}p^{-s} \right] \Bigg).
\end{multline*}
    \end{enumerate}
\end{prop}

\begin{rmk}
   In principle, one can consider $U_{b}^2$ and $\mathcal{H}_b(s)$ for any $b\in\N$ and use \eqref{eqn:MainAnswer3} to compute $\slzeta{L}{s}$ for any arbitrary $B\in\N$. Recall the definition of  $\mathcal{H}_b(s)$ in \eqref{eqn:defn-Y,H}. The most important part for calculating $\mathcal{H}_b(s)$ is  to compute integers $c_b({\bm w})$, which are determined by the sets $\mathcal{T}_t$ and $\mathcal{S}_t$. In general, $\mathcal{S}_t \subseteq \Z_{\ge 0}^{|U_b^{2}|}$ are positive integral points in lattice cosets, which are determined purely the by group structure of $U_b^{2}$. The problem of computing $c_b({\bm w})$ thus connects to additive combinatorial nature of the abelian group $U_b^{2}$, which we expect to be challenging future problem.
\end{rmk}

%\section*{Acknowledgments} We proclaim that all authors equally contributed to this article.

\end{document}